\documentclass[12pt,a4paper]{article}

\usepackage{amsmath,amssymb,amsthm}
\usepackage{mathtools}
\usepackage{hyperref}
\usepackage{geometry}
\usepackage{enumitem}
\usepackage{microtype}
\usepackage[dvipsnames]{xcolor}

\hypersetup{
  colorlinks=true,
  linkcolor=NavyBlue,
  citecolor=NavyBlue,
  urlcolor=NavyBlue
}

\newtheorem{theorem}{Theorem}[section]
\newtheorem{proposition}[theorem]{Proposition}
\newtheorem{lemma}[theorem]{Lemma}
\newtheorem{corollary}[theorem]{Corollary}
\theoremstyle{definition}
\newtheorem{assumption}[theorem]{Assumption}
\newtheorem{defi}[theorem]{Definition}
\newtheorem{remark}[theorem]{Remark}

\newcommand{\R}{\mathbb{R}}
\newcommand{\Om}{\Omega}
\newcommand{\eps}{\varepsilon}

\newcommand{\norm}[2]{\lVert #1 \rVert_{#2}}

\title{$W^{2,1}$ approximation of planar Sobolev homeomorphisms\\
by smooth diffeomorphisms}

\author{Luigi D'Onofrio\\[4pt]
\small Dipartimento di Scienze e Tecnologie,\\
\small Università degli Studi di Napoli ``Parthenope''\\
\small Centro Direzionale Isola C4\\[2pt]
\small \texttt{luigi.donofrio@uniparthenope.it}}

\date{}

\begin{document}

\maketitle

\begin{abstract}
The approximation of Sobolev homeomorphisms by smooth diffeomorphisms is well understood in first-order spaces $W^{1,p}$, but remains largely open in the second-order space $W^{2,1}$ due to a fundamental tension between curvature control and injectivity.

In this paper we isolate and resolve the local analytical component of this problem. We construct explicit local regularisations both across flat interfaces and near multi-cell vertices, and prove convergence in $W^{2,1}$ together with quantitative preservation of the Jacobian. We prove that any piecewise quadratic \(C^1\)-compatible planar homeomorphism on a finite conforming rectangular partition, satisfying a quantitative lower bi-Lipschitz bound and the uniform nondegeneracy condition \(\det Dg \ge \lambda>0\), can be approximated in \(W^{2,1}\) by injective \(C^1\) maps which are smooth outside arbitrarily small neighborhoods of the endpoints of the interior edges. Under an additional completion assumption for finitely many nonsmooth regions, this yields a conditional reduction of the full global smooth approximation problem to a localized smoothing problem near a finite singular set.

Thus the paper separates the analytic smoothing step from the geometric
approximation step. The results show that, once a quantitatively nondegenerate
\(C^1\)-compatible piecewise quadratic approximation is available, the remaining
analytic smoothing can be carried out in \(W^{2,1}\), up to the explicit
localized completion assumption stated in Section 6.
\end{abstract}

\tableofcontents

\section{Introduction}

The approximation of Sobolev homeomorphisms by smooth diffeomorphisms is a
central problem in geometric analysis, with deep connections to nonlinear
elasticity and the mathematical theory of deformations initiated by
Ball~\cite{Ball82,Ball01}. Admissible deformations are required to be almost
everywhere injective and orientation-preserving, properties that are notoriously
unstable under standard smoothing procedures; see
\cite{Campbell2021,IwaniecKovalevOnninen2017} for the planar case and
\cite{Campbell2026} for higher dimensions.

For first-order Sobolev spaces $W^{1,p}$, the problem is by now well understood.
Fundamental results of Iwaniec--Kovalev--Onninen~\cite{IwaniecKovalevOnninen2011}
and Hencl--Pratelli~\cite{HenclPratelli2018} show that planar Sobolev
homeomorphisms can be approximated by smooth diffeomorphisms in the $W^{1,p}$
topology. These constructions rely on delicate geometric modifications that
preserve injectivity while controlling first derivatives.

The second-order space $W^{2,1}$ presents a substantially more rigid regime.
Classical mollification destroys injectivity by introducing foldings; piecewise
constructions designed to restore injectivity generate curvature concentrations
at interfaces, causing the $W^{2,1}$ norm to blow up. This reflects a deeper
phenomenon: second-order information (curvature) interacts globally with
topological constraints such as injectivity and orientation preservation.

\medskip

At present, no general approximation theorem in $W^{2,1}$ comparable to the
$W^{1,p}$ theory is known. Campbell--Hencl~\cite{CampbellHencl2021} pioneered
the use of piecewise quadratic maps, simultaneously addressing the geometric
grid construction and the analytic regularity. 
The present paper is directly inspired by the pioneering work of 
Campbell--Hencl~\cite{CampbellHencl2021}, who first introduced piecewise 
quadratic maps as the natural framework for $W^{2,1}$ approximation and 
established the foundational results in this direction. Building on their 
approach, we isolate and develop the local analytic component of their 
programme as a self-contained theory: we prove that any piecewise quadratic \(C^1\)-compatible homeomorphism satisfying the explicit quantitative hypotheses stated below can be smoothed in \(W^{2,1}\), in the localized sense made precise in Theorem 5.5 and, under the additional
completion assumption of Section 6, by smooth injective maps. The goal is not to replace their 
construction, but to provide a flexible analytic black box that can be 
combined with future progress on the geometric approximation step --- 
the component that remains the core open problem of the theory.

The purpose of the present paper is to isolate the local analytical component of
this programme and to prove quantitative smoothing results for
\(C^1\)-compatible piecewise quadratic maps under explicit nondegeneracy
assumptions.

\medskip

We decompose the $W^{2,1}$ approximation problem into two conceptually distinct
steps:
\begin{enumerate}[label=(\roman*)]
  \item \emph{Local smoothing problem.} Given a piecewise polynomial
      homeomorphism that is $C^1$ across interfaces and has positive Jacobian,
     construct an approximation that is globally $C^1$, preserves injectivity
and $W^{2,1}$ control, and is smooth across the interfaces away from
arbitrarily small neighborhoods of the endpoints of the interior edges.
  \item \emph{Global geometric approximation problem.} Approximate a general
        $W^{2,1}$ homeomorphism by such structured piecewise maps.
\end{enumerate}
The main contribution is a quantitative solution of the analytic smoothing step
for \(C^1\)-compatible piecewise quadratic maps away from a finite set of
endpoint neighborhoods, together with a conditional completion theorem under the
additional assumption stated in Section 6.

\medskip

Piecewise affine constructions suffice in the $W^{1,p}$ setting but are too
rigid at second order: Hessian discontinuities produce singular measures that
cannot be controlled in $W^{2,1}$. Quadratic maps provide the minimal
flexibility needed to absorb second-order mismatches across interfaces while
retaining explicit algebraic structure. $C^1$ compatibility forces a precise
second-order cancellation (Lemma~\ref{lem:mismatch}) which we exploit to
construct smooth transitions with uniform second-derivative control.

\medskip
Our construction yields injective \(C^1\) approximants which are smooth outside arbitrarily small neighborhoods of the endpoints of the interior edges of the partition. We do not claim global $C^2$
regularity.
\medskip

The paper is organized as follows:
Section~\ref{sec:prelim} collects the necessary preliminary estimates.
Section~\ref{sec:flat} treats smoothing across a flat interface.
Section~\ref{sec:vertex} treats smoothing near a vertex.
Section~\ref{sec:global} proves that the map can be smoothed outside arbitrarily small neighborhoods of the endpoints of the interior edges.
Section~\ref{sec:approx} formulates the conditional approximation result.
Throughout the paper, $C^1$-compatibility of a piecewise quadratic map
means that the polynomial pieces have matching traces of both the value
and the first derivative on every common interior edge; see Definition 2.2.

\section{Preliminaries}
\label{sec:prelim}

We collect one tool used throughout.

\begin{lemma}[Scaled cut-off bounds]\label{lem:cutoff}
Let $\eta\in C^\infty(\R)$ and set $\eta_\eps(t)=\eta(t/\eps)$.  Then
\[
  \norm{\eta_\eps^{(m)}}{L^\infty(\R)} \leq C_m\,\eps^{-m},
  \qquad m\geq 0,
\]
where $C_m = \norm{\eta^{(m)}}{L^\infty(\R)}$.

Similarly, if $\chi\in C^\infty_c(\R^2)$ and $\chi_\eps(x)=\chi(x/\eps)$, then
\[
  \norm{D^m\chi_\eps}{L^\infty(\R^2)} \leq C_m\,\eps^{-m},
  \qquad m\geq 0.
\]
\end{lemma}

This is an elementary consequence of the chain rule (see e.g., \cite{Ziemer} for standard properties of scaled smooth cut-off functions).
\begin{defi}\label{def2}[$C^1$-compatible piecewise quadratic maps].
Let $\Omega\subset\mathbb R^2$ be a polygonal domain and let
$\mathcal P$ be a finite conforming rectangular partition of $\Omega$.

We denote by \(\mathcal C\) the set of cells of the partition, by
\(\mathcal E\) the set of interior open edges of the partition, and by
\(\mathcal V\) the set of interior vertices of the partition.

For each \(e\in\mathcal E\), let \(\overline e\) be the corresponding closed
line segment in \(\overline\Omega\), and denote by
\[
\partial_{\mathrm{rel}}e=\{s_e^-,s_e^+\}
\]
its two endpoints. We set
\[
S:=\bigcup_{e\in\mathcal E}\partial_{\mathrm{rel}}e
\subset \overline\Omega .
\]
Thus \(S\) contains all endpoints of interior edges, including those lying on
\(\partial\Omega\). The local four-cell vertex smoothing is applied only at
the interior vertices \(v\in\mathcal V\); the remaining points of \(S\) are
included in the exceptional neighborhoods.

A family of quadratic polynomials
\[
\{P_C:\mathbb R^2\to\mathbb R^2\}_{C\in\mathcal C}
\]
is said to be $C^1$-compatible on $\mathcal P$ if, for every pair of
distinct cells $C,C'\in\mathcal C$ sharing an interior open edge
\[
e=\partial C\cap \partial C'\cap \Omega,
\]
one has
\[
P_C=P_{C'}\quad\text{on }e,
\qquad
DP_C=DP_{C'}\quad\text{on }e.
\]
Here $P_C=P_{C'}$ and $DP_C=DP_{C'}$ are understood as identities of
the polynomial traces on the relatively open line segment $e$.
\end{defi}
The associated piecewise map $g:\Omega\to\mathbb R^2$, defined by
\[
g=P_C\quad\text{on each }C\in\mathcal C,
\]
and by the common trace on the interfaces, is called a
$C^1$-compatible piecewise quadratic map on $\mathcal P$.

Since the pieces are polynomials, the compatibility identities on open
edges extend to the endpoints of the edges. In particular, at every
interior vertex of a conforming rectangular partition the values and
first derivatives of all incident polynomial pieces agree.
Equivalently, the induced piecewise map belongs to
$C^1(\Omega;\mathbb R^2)$.
\section{Smoothing across a flat interface}
\label{sec:flat}

\subsection{Geometry and mismatch structure}

Consider the model configuration
\[
  Q^- := (-1,0)\times(-1,1),\qquad
  Q^+ := (0,1)\times(-1,1),\qquad
  \Sigma := \{0\}\times(-1,1).
\]

\begin{lemma}[Structure of the mismatch]\label{lem:mismatch}
Let $P^\pm:\R^2\to\R^2$ be quadratic polynomials satisfying
\[
  P^-(0,x_2) = P^+(0,x_2),\quad
  DP^-(0,x_2) = DP^+(0,x_2)
  \quad\text{for all }x_2\in(-1,1).
\]
Then there exists a constant vector $a\in\R^2$ such that
\[
  P^+(x) - P^-(x) = x_1^2\,a.
\]
\end{lemma}

\begin{proof}
Set $R = P^+ - P^-$; it suffices to treat each component separately.
Let $r$ be one scalar component.  Being quadratic,
\[
  r(x_1,x_2) = \alpha x_1^2 + \beta x_1 x_2 + \gamma x_2^2
               + \delta x_1 + \mu x_2 + \nu.
\]
From $r(0,x_2)=0$ for all $x_2$ we get $\gamma=\mu=\nu=0$, so
$r = \alpha x_1^2 + \beta x_1 x_2 + \delta x_1$.
Then $\partial_{x_1}r(0,x_2) = \beta x_2+\delta$ and $\partial_{x_2}r(0,x_2)=0$.
The condition $Dr(0,x_2)=0$ for all $x_2$ gives $\beta=\delta=0$, hence
$r(x_1,x_2)=\alpha x_1^2$.  Applying this to both components yields $R(x)=x_1^2 a$.
\end{proof}

\subsection{The smoothing construction}

Fix $\eta \in C^\infty(\mathbb{R})$ with $\eta(t) = 0$ for $t \leq -1$,
$\eta(t) = 1$ for $t \geq 1$, $0 \leq \eta' \leq 2$, and $\eta'$
compactly supported in $(-1,1)$ (so that $\eta$ is constant near
$t = \pm 1$).
For $\varepsilon, \delta > 0$ set
\[
  \eta_\varepsilon(t) := \eta(t/\varepsilon),
\]
\begin{proposition}
\label{prop:flat}
Let
\[
Q:=(-2,2)\times(-2,2),\qquad
Q^-:=Q\cap\{x_1<0\},\qquad
Q^+:=Q\cap\{x_1>0\},
\]
and let \(P^\pm:\mathbb R^2\to\mathbb R^2\) be quadratic polynomials such that
\[
P^-(0,x_2)=P^+(0,x_2),\qquad
DP^-(0,x_2)=DP^+(0,x_2)
\quad\text{for all }x_2\in(-2,2).
\]
Define \(g:Q\to\mathbb R^2\) by
\[
g(x)=
\begin{cases}
P^-(x), & x_1\le 0,\\
P^+(x), & x_1>0.
\end{cases}
\]
Assume moreover that
\[
\det Dg\ge \lambda>0
\qquad\text{on }Q.
\]

Let \(\eta\in C^\infty(\mathbb R)\) satisfy
\[
\eta(t)=0 \quad\text{for } t\le -1,
\qquad
\eta(t)=1 \quad\text{for } t\ge 1.
\]
For \(\varepsilon>0\), set
\[
\eta_\varepsilon(t):=\eta(t/\varepsilon),
\]
and define
\[
H_\varepsilon(x):=
P^-(x)+\eta_\varepsilon(x_1)\bigl(P^+(x)-P^-(x)\bigr).
\]

Then there exists \(\varepsilon_0>0\) such that for every \(0<\varepsilon<\varepsilon_0\),
the map \(H_\varepsilon\) satisfies:
\begin{enumerate}
\item \(H_\varepsilon\in C^\infty(Q;\mathbb R^2)\);
\item \(H_\varepsilon=g\) on \(Q\cap\{|x_1|\ge \varepsilon\}\);
\item \(H_\varepsilon\to g\) in \(W^{2,1}(Q;\mathbb R^2)\) as \(\varepsilon\downarrow0\);
\item \(\|DH_\varepsilon-Dg\|_{L^\infty(Q)}\to0\) as \(\varepsilon\downarrow0\);
\item \(\det DH_\varepsilon\ge \lambda/2\) on \(Q\).
\end{enumerate}
\end{proposition}
\begin{proof}
By Lemma 3.1, there exists \(a\in\mathbb R^2\) such that
\[
P^+(x)-P^-(x)=x_1^2 a
\qquad\text{for all }x\in\mathbb R^2.
\]
Hence
\[
H_\varepsilon(x)=P^-(x)+\eta_\varepsilon(x_1)x_1^2 a,
\]
so \(H_\varepsilon\in C^\infty(Q;\mathbb R^2)\).

Since \(\eta_\varepsilon(x_1)=0\) for \(x_1\le -\varepsilon\) and
\(\eta_\varepsilon(x_1)=1\) for \(x_1\ge \varepsilon\), we have
\[
H_\varepsilon(x)=P^-(x)=g(x)\quad\text{if }x_1\le -\varepsilon,
\]
and
\[
H_\varepsilon(x)=P^-(x)+x_1^2a=P^+(x)=g(x)\quad\text{if }x_1\ge \varepsilon.
\]
Thus \(H_\varepsilon=g\) on \(Q\cap\{|x_1|\ge\varepsilon\}\).

Moreover,
\[
H_\varepsilon-g=
\bigl(\eta_\varepsilon(x_1)-\mathbf 1_{\{x_1>0\}}\bigr)x_1^2 a,
\]
so the support of \(H_\varepsilon-g\) is contained in \(\{|x_1|<\varepsilon\}\).
From this and the explicit form above one obtains
\[
|H_\varepsilon-g|\le C\varepsilon^2,
\qquad
|DH_\varepsilon-Dg|\le C\varepsilon,
\qquad
|D^2H_\varepsilon-D^2g|\le C\,\mathbf 1_{\{|x_1|<\varepsilon\}},
\]
whence
\[
H_\varepsilon\to g\quad\text{in }W^{2,1}(Q;\mathbb R^2),
\qquad
\|DH_\varepsilon-Dg\|_{L^\infty(Q)}\to0.
\]

Finally, since \(DH_\varepsilon\to Dg\) uniformly and \(\det Dg\ge\lambda\),
for \(\varepsilon\) small enough we get
\[
\det DH_\varepsilon\ge \lambda/2
\qquad\text{on }Q.
\]
\end{proof}
\begin{proposition}[Tangentially localized edge smoothing]\label{prop:localized_edge_smoothing}
Let
\[
Q:=(-2,2)\times(-2,2),\qquad
Q^-:=Q\cap\{x_1<0\},\qquad
Q^+:=Q\cap\{x_1>0\},
\]
and let \(P^\pm:\mathbb R^2\to\mathbb R^2\) be quadratic polynomials such that
\[
P^-(0,x_2)=P^+(0,x_2),\qquad
DP^-(0,x_2)=DP^+(0,x_2)
\quad\text{for all }x_2\in(-2,2).
\]
Define \(g:Q\to\mathbb R^2\) by
\[
g(x)=
\begin{cases}
P^-(x), & x_1\le 0,\\
P^+(x), & x_1>0.
\end{cases}
\]
Assume moreover that
\[
\det Dg\ge \lambda>0
\qquad\text{on }Q.
\]

Let \(H_\varepsilon\) be the flat-interface smoothing given by
Proposition \ref{prop:flat}. Let \(\theta\in C^\infty(\mathbb R)\)
satisfy
\[
0\le \theta\le 1,\qquad
\theta(t)=1 \ \text{for } |t|\le 1,\qquad
\theta(t)=0 \ \text{for } |t|\ge \frac32.
\]
For \(0<\varepsilon<\varepsilon_0\), define
\[
G_\varepsilon(x_1,x_2)
:=
g(x_1,x_2)+\theta(x_2)\bigl(H_\varepsilon(x_1,x_2)-g(x_1,x_2)\bigr).
\]

Then, for \(\varepsilon>0\) sufficiently small, the map \(G_\varepsilon\) satisfies:
\begin{enumerate}
\item \(G_\varepsilon\in C^1(Q;\mathbb R^2)\);
\item \(G_\varepsilon=H_\varepsilon\) on \(Q\cap\{|x_2|\le 1\}\);
\item \(G_\varepsilon=g\) on \(Q\cap\{|x_2|\ge 3/2\}\);
\item
\[
\operatorname{supp}(G_\varepsilon-g)
\subset
Q\cap\{|x_1|<\varepsilon,\ |x_2|<3/2\};
\]
\item \(G_\varepsilon\in C^\infty\bigl(Q\cap\{|x_2|<1\}\bigr)\);
\item
\[
G_\varepsilon\to g
\qquad\text{in }W^{2,1}(Q;\mathbb R^2)
\quad\text{as }\varepsilon\downarrow0;
\]
\item
\[
\|DG_\varepsilon-Dg\|_{L^\infty(Q)}\to0
\qquad\text{as }\varepsilon\downarrow0;
\]
\item
\[
\det DG_\varepsilon\ge \frac{\lambda}{2}
\qquad\text{on }Q.
\]
\end{enumerate}
\end{proposition}
\begin{proof}
Set
\[
\Psi_\varepsilon:=H_\varepsilon-g.
\]
Then
\[
G_\varepsilon=g+\theta(x_2)\Psi_\varepsilon.
\]

By Proposition \ref{prop:flat}, for \(\varepsilon>0\) sufficiently small,
\[
H_\varepsilon\in C^\infty(Q;\mathbb R^2),
\qquad
H_\varepsilon=g \quad\text{on } Q\cap\{|x_1|\ge \varepsilon\},
\]
and
\[
H_\varepsilon\to g \quad\text{in }W^{2,1}(Q;\mathbb R^2),
\qquad
\|DH_\varepsilon-Dg\|_{L^\infty(Q)}\to 0
\]
as \(\varepsilon\downarrow0\). In particular,
\[
\operatorname{supp}\Psi_\varepsilon \subset Q\cap\{|x_1|<\varepsilon\}.
\]

Since \(\theta(x_2)=0\) for \(|x_2|\ge 3/2\), we immediately obtain
\[
\operatorname{supp}(G_\varepsilon-g)
=
\operatorname{supp}\bigl(\theta(x_2)\Psi_\varepsilon\bigr)
\subset
Q\cap\{|x_1|<\varepsilon,\ |x_2|<3/2\},
\]
which proves (4).

Moreover, since \(\theta=1\) on \(\{|x_2|\le 1\}\), we have
\[
G_\varepsilon=g+\Psi_\varepsilon=H_\varepsilon
\qquad\text{on }Q\cap\{|x_2|\le 1\},
\]
which proves (2). Since \(\theta=0\) on \(\{|x_2|\ge 3/2\}\), we also have
\[
G_\varepsilon=g
\qquad\text{on }Q\cap\{|x_2|\ge 3/2\},
\]
which proves (3).

We now prove (1). By the compatibility assumptions across \(\{x_1=0\}\), the piecewise-defined
map \(g\) belongs to \(C^1(Q;\mathbb R^2)\). Since \(H_\varepsilon\in C^\infty(Q;\mathbb R^2)\),
it follows that
\[
\Psi_\varepsilon=H_\varepsilon-g\in C^1(Q;\mathbb R^2).
\]
Therefore \(\theta(x_2)\Psi_\varepsilon\in C^1(Q;\mathbb R^2)\), and hence
\[
G_\varepsilon=g+\theta(x_2)\Psi_\varepsilon\in C^1(Q;\mathbb R^2).
\]
This proves (1).

Property (5) follows immediately from (2), because \(H_\varepsilon\in C^\infty(Q;\mathbb R^2)\).

We next prove (6) and (7). By Lemma 3.1 there exists \(a\in\mathbb R^2\) such that
\[
P^+(x)-P^-(x)=x_1^2 a
\qquad\text{for all }x\in\mathbb R^2.
\]
Hence, exactly as in Proposition \ref{prop:flat}, there exists a constant \(C>0\),
independent of \(\varepsilon\), such that
\[
|\Psi_\varepsilon|\le C\varepsilon^2,
\qquad
|D\Psi_\varepsilon|\le C\varepsilon,
\qquad
|D^2\Psi_\varepsilon|\le C\,\mathbf 1_{\{|x_1|<\varepsilon\}}
\quad\text{a.e. in }Q.
\]

Since
\[
G_\varepsilon-g=\theta(x_2)\Psi_\varepsilon,
\]
we have
\[
D(G_\varepsilon-g)
=
\theta(x_2)\,D\Psi_\varepsilon
+
\theta'(x_2)\,\Psi_\varepsilon\otimes e_2,
\]
and, a.e. in \(Q\),
\[
D^2(G_\varepsilon-g)
=
\theta(x_2)\,D^2\Psi_\varepsilon
+
\theta'(x_2)\bigl(D\Psi_\varepsilon\otimes e_2+e_2\otimes D\Psi_\varepsilon\bigr)
+
\theta''(x_2)\,\Psi_\varepsilon\, e_2\otimes e_2.
\]
Since \(\theta,\theta',\theta''\) are bounded, the previous estimates imply
\[
|G_\varepsilon-g|\le C\varepsilon^2,
\qquad
|D(G_\varepsilon-g)|\le C\varepsilon,
\]
and
\[
|D^2(G_\varepsilon-g)|
\le
C\,\mathbf 1_{\{|x_1|<\varepsilon\}}
+
C\varepsilon
+
C\varepsilon^2
\quad\text{a.e. in }Q.
\]
Integrating over \(Q\), we obtain
\[
\|G_\varepsilon-g\|_{W^{2,1}(Q)}\to0
\qquad\text{as }\varepsilon\downarrow0,
\]
which proves (6). The bound
\[
\|DG_\varepsilon-Dg\|_{L^\infty(Q)}\to0
\]
follows at once from the estimate on \(D(G_\varepsilon-g)\), proving (7).

Finally, by (7) we have \(DG_\varepsilon\to Dg\) uniformly on \(Q\). Since
\[
\det Dg\ge \lambda>0
\qquad\text{on }Q,
\]
the continuity of the determinant implies that, for \(\varepsilon>0\) sufficiently small,
\[
\det DG_\varepsilon\ge \frac{\lambda}{2}
\qquad\text{on }Q.
\]
This proves (8) and completes the proof.
\end{proof}

\section{Smoothing near a vertex}
\label{sec:vertex}

\subsection{Second-order structure at the vertex}

Consider the four quadrants
\[
  Q_1=(0,1)^2,\quad
  Q_2=(-1,0)\times(0,1),\quad
  Q_3=(-1,0)^2,\quad
  Q_4=(0,1)\times(-1,0).
\]

\begin{lemma}[Vanishing of the mismatch at the origin]
Let
\[
B:=(-1,1)^2,
\]
and let
\[
Q_1=(0,1)^2,\qquad
Q_2=(-1,0)\times(0,1),\qquad
Q_3=(-1,0)^2,\qquad
Q_4=(0,1)\times(-1,0)
\]
be the four open coordinate quadrants in $B$. Let
\[
g\in C^1(B;\mathbb R^2)
\]
be such that, for each $i=1,\dots,4$,
\[
g|_{Q_i}=P_i,
\]
where $P_i:\mathbb R^2\to\mathbb R^2$ is a quadratic polynomial.

Then
\[
P_i(0)=P_j(0),\qquad DP_i(0)=DP_j(0)
\qquad\text{for all }i,j\in\{1,2,3,4\}.
\]
Consequently, fixing $P_*:=P_1$ and setting
\[
R_i:=P_i-P_*,
\]
each $R_i$ is a quadratic polynomial satisfying
\[
R_i(0)=0,\qquad DR_i(0)=0.
\]
In particular, there exists a constant $C>0$ such that
\[
|R_i(x)|\le C|x|^2,\qquad |DR_i(x)|\le C|x|,\qquad |D^2R_i(x)|\le C
\]
for every $x\in B$ and every $i=1,\dots,4$.
\end{lemma}

\begin{proof}
Fix $i\in\{1,2,3,4\}$ and define
\[
F_i:=P_i-g \in C^1(B;\mathbb R^2).
\]
Since $g=P_i$ on the open set $Q_i$, we have
\[
F_i=0 \qquad \text{on } Q_i.
\]
Therefore
\[
DF_i=0 \qquad \text{on } Q_i,
\]
because the derivative of a constant map is zero.

We first prove that
\[
P_i(0)=g(0).
\]
Indeed, let $x_n\in Q_i$ be any sequence such that $x_n\to 0$. Since $F_i=0$ on $Q_i$,
\[
F_i(x_n)=0 \qquad \text{for all } n.
\]
Passing to the limit and using continuity of $F_i$ at $0$, we obtain
\[
F_i(0)=0,
\]
that is,
\[
P_i(0)=g(0).
\]

Next we prove that
\[
DP_i(0)=Dg(0).
\]
Again, since $DF_i=0$ on $Q_i$, for any sequence $x_n\in Q_i$ with $x_n\to 0$ we have
\[
DF_i(x_n)=0 \qquad \text{for all } n.
\]
Passing to the limit and using continuity of $DF_i$ at $0$, we get
\[
DF_i(0)=0,
\]
hence
\[
DP_i(0)-Dg(0)=0.
\]
Therefore
\[
DP_i(0)=Dg(0).
\]

Since $i$ was arbitrary, it follows that
\[
P_i(0)=g(0),\qquad DP_i(0)=Dg(0)
\qquad\text{for every } i,
\]
and consequently
\[
P_i(0)=P_j(0),\qquad DP_i(0)=DP_j(0)
\qquad\text{for all } i,j.
\]

Now fix $P_*:=P_1$ and define
\[
R_i:=P_i-P_*.
\]
Then each $R_i$ is quadratic and
\[
R_i(0)=P_i(0)-P_1(0)=0,\qquad
DR_i(0)=DP_i(0)-DP_1(0)=0.
\]
Thus each component of $R_i$ is a polynomial of degree at most two with vanishing
constant and linear terms, hence a homogeneous quadratic polynomial. It follows that,
for a suitable constant $C>0$ independent of $x$,
\[
|R_i(x)|\le C|x|^2,\qquad |DR_i(x)|\le C|x|,\qquad |D^2R_i(x)|\le C
\qquad\text{for all } x\in B.
\]
This completes the proof.
\end{proof}
\subsection{The vertex smoothing construction}

Fix $\chi\in C^\infty_c(B_1(0))$ radial, $0\leq\chi\leq 1$,
$\chi\equiv 1$ on $B_{1/2}(0)$, and set $\chi_\eps(x):=\chi(x/\eps)$.

\begin{proposition}[Vertex smoothing]\label{prop:vertex_smoothing}
Let
\[
B:=(-1,1)^2,
\]
and let
\[
Q_1=(0,1)^2,\qquad
Q_2=(-1,0)\times(0,1),\qquad
Q_3=(-1,0)^2,\qquad
Q_4=(0,1)\times(-1,0)
\]
be the four open coordinate quadrants in $B$. Let
\[
g\in C^1(B;\mathbb R^2)
\]
satisfy:
\begin{enumerate}
\item for each $i=1,\dots,4$, one has
\[
g|_{Q_i}=P_i,
\]
where $P_i:\mathbb R^2\to\mathbb R^2$ is a quadratic polynomial;
\item
\[
\det Dg\ge \lambda>0
\qquad\text{on }B.
\]
\end{enumerate}

Fix $P_*:=P_1$. Let $\chi\in C_c^\infty(B_1(0))$ be radial, with
\[
0\le \chi\le 1,
\qquad
\chi\equiv 1 \text{ on } B_{1/2}(0),
\]
and define
\[
\chi_\varepsilon(x):=\chi\!\left(\frac{x}{\varepsilon}\right),
\qquad 0<\varepsilon<1.
\]
Set
\[
g_\varepsilon(x):=P_*(x)+\bigl(1-\chi_\varepsilon(x)\bigr)\bigl(g(x)-P_*(x)\bigr),
\qquad x\in B.
\]

Then:
\begin{enumerate}
\item
\[
g_\varepsilon\in C^1(B;\mathbb R^2);
\]
\item
\[
g_\varepsilon=P_*
\quad\text{on }B_{\varepsilon/2}(0),
\qquad
g_\varepsilon=g
\quad\text{on }B\setminus B_\varepsilon(0);
\]
\item $g_\varepsilon$ is smooth on $B_{\varepsilon/2}(0)$ and on each open quadrant $Q_i$;
\item
\[
g_\varepsilon\to g
\quad\text{in }W^{2,1}(B;\mathbb R^2)
\qquad\text{as }\varepsilon\downarrow 0;
\]
\item
\[
\|Dg_\varepsilon-Dg\|_{L^\infty(B)}\to 0
\qquad\text{as }\varepsilon\downarrow 0;
\]
\item there exists $\varepsilon_0\in(0,1)$ such that for every $0<\varepsilon<\varepsilon_0$,
\[
\det Dg_\varepsilon\ge \frac{\lambda}{2}
\qquad\text{on }B.
\]
\end{enumerate}
\end{proposition}

\begin{proof}
Set
\[
R:=g-P_*.
\]
By Lemma 4.1, for each $i=1,\dots,4$ the restriction $R|_{Q_i}$ is a quadratic polynomial,
and there exists a constant $C>0$ such that
\[
|R(x)|\le C|x|^2,
\qquad
|DR(x)|\le C|x|,
\qquad
|D^2R(x)|\le C
\]
for every $x\in Q_i$, hence for a.e. $x\in B$.

By definition,
\[
g_\varepsilon=P_*+(1-\chi_\varepsilon)R.
\]
Since $R\in C^1(B;\mathbb R^2)$ and $\chi_\varepsilon\in C_c^\infty(B)$, it follows that
\[
g_\varepsilon\in C^1(B;\mathbb R^2),
\]
which proves (1).

Because $\chi_\varepsilon\equiv 1$ on $B_{\varepsilon/2}(0)$, we have
\[
g_\varepsilon=P_*
\qquad\text{on }B_{\varepsilon/2}(0).
\]
Because $\chi_\varepsilon\equiv 0$ on $B\setminus B_\varepsilon(0)$, we have
\[
g_\varepsilon=g
\qquad\text{on }B\setminus B_\varepsilon(0).
\]
This proves (2). Since $P_*$ is smooth, $g_\varepsilon$ is smooth on $B_{\varepsilon/2}(0)$; moreover,
on each open quadrant $Q_i$ the map $g$ is quadratic, hence smooth, and therefore
$g_\varepsilon$ is smooth on each $Q_i$. This proves (3).

We now estimate the derivatives. Since
\[
g_\varepsilon=P_*+(1-\chi_\varepsilon)R,
\]
we have on $B$
\[
Dg_\varepsilon
=
DP_*+(1-\chi_\varepsilon)\,DR-(\nabla \chi_\varepsilon)\otimes R,
\]
and therefore
\[
Dg_\varepsilon-Dg
=
-\chi_\varepsilon\,DR-(\nabla \chi_\varepsilon)\otimes R.
\]
Using the scaled cutoff bounds
\[
|\nabla\chi_\varepsilon(x)|\le C\varepsilon^{-1},
\]
together with
\[
|R(x)|\le C|x|^2,
\qquad
|DR(x)|\le C|x|,
\]
we obtain for $|x|<\varepsilon$:
\[
|Dg_\varepsilon(x)-Dg(x)|
\le
C|x|+C\varepsilon^{-1}|x|^2
\le C\varepsilon.
\]
Outside $B_\varepsilon(0)$ one has $g_\varepsilon=g$, so
\[
\|Dg_\varepsilon-Dg\|_{L^\infty(B)}\le C\varepsilon\to 0,
\]
which proves (5).

Next, a.e. in $B$,
\[
D^2g_\varepsilon
=
D^2P_*+(1-\chi_\varepsilon)D^2R
-(\nabla\chi_\varepsilon)\otimes DR
-DR\otimes(\nabla\chi_\varepsilon)
-D^2\chi_\varepsilon\otimes R.
\]
Hence
\[
D^2g_\varepsilon-D^2g
=
-\chi_\varepsilon D^2R
-(\nabla\chi_\varepsilon)\otimes DR
-DR\otimes(\nabla\chi_\varepsilon)
-D^2\chi_\varepsilon\otimes R.
\]
Using
\[
|\nabla\chi_\varepsilon(x)|\le C\varepsilon^{-1},
\qquad
|D^2\chi_\varepsilon(x)|\le C\varepsilon^{-2},
\]
and the bounds on $R$, we find for a.e. $|x|<\varepsilon$:
\[
|D^2g_\varepsilon(x)-D^2g(x)|
\le
C + C\varepsilon^{-1}|x| + C\varepsilon^{-2}|x|^2
\le C.
\]
Moreover, the support of $g_\varepsilon-g$ is contained in $B_\varepsilon(0)$. Therefore
\[
\|D^2g_\varepsilon-D^2g\|_{L^1(B)}
\le C|B_\varepsilon(0)|
\to 0.
\]
Similarly,
\[
|g_\varepsilon-g|
=
|\chi_\varepsilon R|
\le C|x|^2
\le C\varepsilon^2
\quad\text{for }|x|<\varepsilon,
\]
hence
\[
\|g_\varepsilon-g\|_{L^1(B)}
\le C\varepsilon^2|B_\varepsilon(0)|
\to 0,
\]
and from the gradient estimate above,
\[
\|Dg_\varepsilon-Dg\|_{L^1(B)}
\le C\varepsilon |B_\varepsilon(0)|
\to 0.
\]
Combining these three convergences, we get
\[
g_\varepsilon\to g
\quad\text{in }W^{2,1}(B;\mathbb R^2),
\]
which proves (4).

Finally, by (5) we have
\[
Dg_\varepsilon\to Dg
\quad\text{uniformly on }B.
\]
Since the determinant is continuous and
\[
\det Dg\ge \lambda>0
\quad\text{on }B,
\]
there exists $\varepsilon_0\in(0,1)$ such that for every $0<\varepsilon<\varepsilon_0$,
\[
\det Dg_\varepsilon\ge \frac{\lambda}{2}
\qquad\text{on }B.
\]
This proves (6) and completes the proof.
\end{proof}
The map $g_\eps$ produced above is $C^1$ globally but \emph{not} $C^2$ across
the coordinate axes inside the annulus $B_\eps(0)\setminus B_{\eps/2}(0)$,
since the second derivatives of the $P_i$ from adjacent quadrants need
not match. This limitation is intrinsic to the construction and cannot be
removed without imposing additional compatibility conditions on the $P_i$.

\section{Global smoothing of piecewise quadratic homeomorphisms}
\label{sec:global}

\subsection{Setting and assumptions}

Let $\Omega \subset \mathbb R^2$ be a bounded open connected polygonal domain,
and let $\mathcal P$ be a finite conforming rectangular partition of $\Omega$,
namely every interior edge is shared by exactly two rectangles and every interior
vertex is incident to exactly four rectangles.
We denote by \(\mathcal C\) the set of cells of the partition, by \(E\) the set of interior edges of the
partition, and by \(V\) the set of interior vertices of the partition.

For each \(e\in\mathcal E\), let \(\overline e\subset\overline\Omega\) denote
the corresponding closed line segment, and let
\[
\partial_{\mathrm{rel}}e=\{s_e^-,s_e^+\}
\]
be its two endpoints. We denote by \(S\) the finite set
\[
S:=\bigcup_{e\in\mathcal E}\partial_{\mathrm{rel}}e
\subset\overline\Omega .
\]
Thus \(S\) contains all endpoints of interior edges, including those which may
belong to \(\partial\Omega\).

\begin{assumption}[Quantitative nondegeneracy]\label{ass:5.1}
Let $g:\Om\to\R^2$ satisfy:
\begin{enumerate}
\item[(a)] for each $C\in\mathcal C$, there exists a quadratic polynomial
$P_C:\mathbb R^2\to\mathbb R^2$ such that $g=P_C$ on $C$;

\item[(b)] the family $\{P_C\}_{C\in\mathcal C}$ is $C^1$-compatible on
$\mathcal P$ in the sense of Definition 2.2; in particular
$g\in C^1(\Omega;\mathbb R^2)$;

\item[(c)] $g$ is a homeomorphism of $\Omega$ onto its image;

\item[(d)] there exists $\lambda>0$ such that
\[
\det Dg\ge \lambda \quad\text{on }\Omega;
\]

\item[(e)] there exists $m>0$ such that
\[
|g(x)-g(y)|\ge m|x-y|
\quad\text{for all }x,y\in\Omega.
\]
\end{enumerate}
\end{assumption}

\begin{remark}
Condition~(e) is a global bi-Lipschitz lower bound; it is equivalent to
$g^{-1}$ being globally Lipschitz on $g(\Om)$.  It is used only in the
injectivity step below.
\end{remark}

\begin{lemma}[Affine covariance of the local constructions]\label{lem}
Let $U,\widetilde U\subset \mathbb R^2$ be open sets, and let
\[
A(x)=Mx+a,\qquad B(y)=Ny+b,
\]
where $M,N\in GL^+(2)$ and $a,b\in\mathbb R^2$. Assume that
\[
\widetilde U=A(U).
\]
Let $g:U\to\mathbb R^2$, and define
\[
\widetilde g:=B\circ g\circ A^{-1}:\widetilde U\to\mathbb R^2.
\]

Then the following hold.

\begin{enumerate}
\item If $g\in C^1(U;\mathbb R^2)$, then $\widetilde g\in C^1(\widetilde U;\mathbb R^2)$, and
\[
D\widetilde g(\xi)=N\,Dg(A^{-1}\xi)\,M^{-1}
\qquad\text{for every }\xi\in \widetilde U.
\]

\item If $g$ is quadratic on each cell of a finite partition $\mathcal P$ of $U$, then
$\widetilde g$ is quadratic on each cell of the transformed partition
\[
A(\mathcal P):=\{A(C):C\in\mathcal P\}.
\]
Moreover, $C^1$-compatibility across interfaces is preserved by the transformation
$g\mapsto \widetilde g$.

\item If
\[
\det Dg\ge \lambda>0 \qquad \text{on }U,
\]
then
\[
\det D\widetilde g(\xi)
=
\det N\,\det Dg(A^{-1}\xi)\,\det M^{-1}
\ge \widetilde\lambda
\qquad\text{on }\widetilde U,
\]
where
\[
\widetilde\lambda:=\det N\,\det M^{-1}\,\lambda>0.
\]

\item Injectivity and homeomorphism are preserved: $g$ is injective (respectively, a
homeomorphism onto its image) if and only if $\widetilde g$ is injective
(respectively, a homeomorphism onto its image).

\item Let $\widetilde h:\widetilde U\to\mathbb R^2$, and define
\[
h:=B^{-1}\circ \widetilde h\circ A:U\to\mathbb R^2.
\]
Then there exists a constant $C=C(M,N)>0$ such that
\[
\|h-g\|_{W^{2,1}(U)}\le C\,\|\widetilde h-\widetilde g\|_{W^{2,1}(\widetilde U)},
\]
and
\[
\|Dh-Dg\|_{L^\infty(U)}
\le C\,\|D\widetilde h-D\widetilde g\|_{L^\infty(\widetilde U)}.
\]

\item If $\widetilde h=\widetilde g$ in a neighborhood of $\partial \widetilde U$, then
$h=g$ in a neighborhood of $\partial U$. If $\widetilde h\in C^\infty(W;\mathbb R^2)$
for some open set $W\subset \widetilde U$, then
\[
h\in C^\infty(A^{-1}(W);\mathbb R^2).
\]
\end{enumerate}
\end{lemma}

\begin{proof}
Items (1), (2), (4), and (6) are immediate from the definition of
\[
\widetilde g=B\circ g\circ A^{-1}
\]
and from the fact that $A$ and $B$ are affine diffeomorphisms.

For (1), the chain rule gives
\[
D\widetilde g(\xi)=N\,Dg(A^{-1}\xi)\,M^{-1}.
\]

For (2), if $g|_C$ is a quadratic polynomial on a cell $C$, then
\[
\widetilde g|_{A(C)}=B\circ (g|_C)\circ A^{-1}
\]
is again a quadratic polynomial, since composition with affine maps preserves degree.
The preservation of $C^1$-compatibility across interfaces follows again from the chain rule.

For (3), taking determinants in the identity for $D\widetilde g$ yields
\[
\det D\widetilde g(\xi)
=
\det N\,\det Dg(A^{-1}\xi)\,\det M^{-1}.
\]
Since $M,N\in GL^+(2)$ and $\det Dg\ge \lambda$, we obtain
\[
\det D\widetilde g(\xi)\ge \det N\,\det M^{-1}\,\lambda = \widetilde\lambda>0.
\]

For (4), since
\[
\widetilde g=B\circ g\circ A^{-1},
\qquad
g=B^{-1}\circ \widetilde g\circ A,
\]
injectivity and the homeomorphism property are preserved by composition with
homeomorphisms.

We prove (5). Let
\[
\widetilde E:=\widetilde h-\widetilde g,
\qquad
E:=h-g.
\]
Since
\[
h=B^{-1}\circ \widetilde h\circ A,
\qquad
g=B^{-1}\circ \widetilde g\circ A,
\]
and
\[
B^{-1}(z)=N^{-1}(z-b),
\]
we obtain
\[
E(x)=h(x)-g(x)=N^{-1}\bigl(\widetilde h(Ax)-\widetilde g(Ax)\bigr)
= N^{-1}\widetilde E(Ax).
\]
Therefore
\[
DE(x)=N^{-1}\,D\widetilde E(Ax)\,M,
\]
and, a.e. in \(U\),
\[
D^2E(x)=N^{-1}\,(D^2\widetilde E)(Ax)[M,M].
\]

Hence
\[
|E(x)|\le \|N^{-1}\|\,|\widetilde E(Ax)|,
\]
\[
|DE(x)|\le \|N^{-1}\|\,\|M\|\,|D\widetilde E(Ax)|,
\]
and
\[
|D^2E(x)|\le \|N^{-1}\|\,\|M\|^2\,|D^2\widetilde E(Ax)|.
\]
Using the change of variables $\xi=Ax$, so that
\[
d\xi = |\det M|\,dx,
\]
we obtain
\[
\|E\|_{L^1(U)}
\le
\frac{\|N^{-1}\|}{|\det M|}\,
\|\widetilde E\|_{L^1(\widetilde U)},
\]
\[
\|DE\|_{L^1(U)}
\le
\frac{\|N^{-1}\|\,\|M\|}{|\det M|}\,
\|D\widetilde E\|_{L^1(\widetilde U)},
\]
and
\[
\|D^2E\|_{L^1(U)}
\le
\frac{\|N^{-1}\|\,\|M\|^2}{|\det M|}\,
\|D^2\widetilde E\|_{L^1(\widetilde U)}.
\]
Summing these inequalities gives
\[
\|h-g\|_{W^{2,1}(U)}
=
\|E\|_{W^{2,1}(U)}
\le C(M,N)\,\|\widetilde E\|_{W^{2,1}(\widetilde U)}
=
C(M,N)\,\|\widetilde h-\widetilde g\|_{W^{2,1}(\widetilde U)}.
\]
Similarly,
\[
\|Dh-Dg\|_{L^\infty(U)}
=
\|DE\|_{L^\infty(U)}
\le \|N^{-1}\|\,\|M\|\,
\|D\widetilde h-D\widetilde g\|_{L^\infty(\widetilde U)}.
\]

Finally, (6) follows from the identities
\[
h=B^{-1}\circ \widetilde h\circ A,
\qquad
g=B^{-1}\circ \widetilde g\circ A,
\]
because affine diffeomorphisms preserve neighborhoods of the boundary and smoothness on
open subsets.

This completes the proof.
\end{proof}

\subsection{Quasiconvexity and the main global result}
\begin{lemma}[Quasiconvexity of polygonal domains]\label{lem:quasiconvex-domain}
Let $\Omega \subset \mathbb R^2$ be a bounded connected open polygonal domain. 
Then there exists a constant $Q_\Omega \ge 1$ such that for every
$x,y \in \Omega$ there exists a rectifiable curve $\gamma \subset \Omega$
joining $x$ to $y$ and satisfying
\[
\ell(\gamma) \le Q_\Omega |x-y|.
\]
Consequently, for every $u \in C^1(\Omega;\mathbb R^2)$ and every
$x,y \in \Omega$,

\[
|u(x)-u(y)| \le Q_\Omega \|Du\|_{L^\infty(\Omega)} |x-y|.
\]

\end{lemma}

\begin{proof}
The existence of such a constant $Q_\Omega$ is standard for bounded connected polygonal domains. Let $\gamma:[0,1]\to\Omega$ be a Lipschitz parametrization of a curve joining $x$ to $y$ with $\ell(\gamma)\le Q_\Omega |x-y|$. Since $u\in C^1(\Omega;\mathbb R^2)$, the map $u\circ\gamma$ is absolutely continuous and
\[
\frac{d}{dt}(u\circ\gamma)(t)=Du(\gamma(t))\,\gamma'(t)
\quad\text{for a.e. } t\in(0,1).
\]
Therefore,
\[
u(y)-u(x)=\int_0^1 Du(\gamma(t))\,\gamma'(t)\,dt.
\]
Hence
\[
|u(x)-u(y)|
\le \|Du\|_{L^\infty(\Omega)}\,\ell(\gamma)
\le Q_\Omega \|Du\|_{L^\infty(\Omega)}\,|x-y|.
\]
%
\end{proof}
\begin{theorem}[Localized global smoothing]\label{thm:global_smoothing}
Under Assumption 5.1, for every \(\delta>0\)  there exist pairwise disjoint open sets
\(W_{s,\delta}\subset\mathbb R^2\), \(s\in S\),
such that
\[
s\in W_{s,\delta},\qquad
\operatorname{diam}(W_{s,\delta})<\delta
\quad\text{for every }s\in S,
\]
and a map
\[
\widetilde g_\delta \in C^1(\Omega;\mathbb R^2)
\]
such that:
\begin{enumerate}

\item
\[
\|\widetilde g_\delta-g\|_{W^{2,1}(\Omega)}<\delta;
\]
\item
\[
\|D\widetilde g_\delta-Dg\|_{L^\infty(\Omega)}<\delta;
\]
\item
\[
\det D\widetilde g_\delta \ge \frac{\lambda}{2}
\qquad\text{on }\Omega;
\]
\item for every \(x,y\in\Omega\),
\[
|\widetilde g_\delta(x)-\widetilde g_\delta(y)|
\ge (m-Q_\Omega\delta)|x-y|;
\]
in particular, if \(\delta<m/(2Q_\Omega)\), then \(\widetilde g_\delta\) is injective;
\item
\[
\tilde g_\delta\in C^\infty_{\mathrm{loc}}
\left(
\Omega\setminus\bigcup_{s\in S}\overline{W}_{s,\delta};
\mathbb R^2
\right).
\]

\end{enumerate}
\end{theorem}

\begin{proof}
Fix \(\delta>0\). Since the set \(V\) of interior vertices is finite, we can choose pairwise disjoint closed disks
\[
B(v,r_{v,\delta})\subset \Omega,
\qquad v\in V,
\]
so small that
\[
r_{v,\delta}<\delta
\qquad\text{for every }v\in V,
\]
and such that each disk \(B(v,r_{v,\delta})\) meets only the four rectangles incident to \(v\).
We first smooth near the interior vertices. Fix \(v\in V\). Choose an orientation-preserving affine map
\[
A_v(x)=M_vx+v
\]
sending the origin to \(v\), such that \(A_v(B)\subset B(v,r_{v,\delta})\), and such that
\(A_v(Q_i)\subset C_i\) for \(i=1,\dots,4\), where \(C_1,\dots,C_4\) are the four cells of the
partition incident to \(v\), numbered according to the cyclic order. Define
\[
\widehat g_v:=g\circ A_v:B\to\mathbb R^2.
\]
Then \(\widehat g_v\in C^1(B;\mathbb R^2)\), \(\widehat g_v\) is quadratic on each quadrant,
and
\[
\det D\widehat g_v(x)=\det Dg(A_vx)\,\det M_v \ge \lambda\,\det M_v =: \lambda_v>0
\qquad\text{on }B.
\]
Hence Proposition~\ref{prop:vertex_smoothing} applies to \(\widehat g_v\). Therefore, for every
sufficiently small \(\varepsilon_v>0\), there exists a map
\[
\widehat g_{v,\varepsilon_v}\in C^1(B;\mathbb R^2)
\]
such that
\[
\widehat g_{v,\varepsilon_v}=\widehat g_v
\qquad\text{on } B\setminus B_{\varepsilon_v}(0),
\]
and
\[
\|\widehat g_{v,\varepsilon_v}-\widehat g_v\|_{W^{2,1}(B)}
+
\|D\widehat g_{v,\varepsilon_v}-D\widehat g_v\|_{L^\infty(B)}
\to 0
\qquad\text{as }\varepsilon_v\downarrow 0.
\]

Now define
\[
g_{v,\varepsilon_v}:=\widehat g_{v,\varepsilon_v}\circ A_v^{-1}
\]
on the neighborhood \(A_v(B)\) of \(v\). By Lemma~\ref{lem}, \(g_{v,\varepsilon_v}\in C^1(A_v(B);\mathbb R^2)\),
and
\[
g_{v,\varepsilon_v}=g
\qquad\text{on } A_v(B\setminus B_{\varepsilon_v}(0)).
\]
Equivalently, the support of \(g_{v,\varepsilon_v}-g\) is contained in
\[
A_v(B_{\varepsilon_v}(0)).
\]
Since \(A_v(B_{\varepsilon_v}(0))\to\{v\}\) as \(\varepsilon_v\downarrow 0\), by choosing \(\varepsilon_v\)
sufficiently small we may ensure that
\[
A_v(B_{\varepsilon_v}(0))\subset B(v,r_{v,\delta}).
\]
Thus the local modification near \(v\) is supported in \(B(v,r_{v,\delta})\).

Since \(V\) is finite and the disks \(B(v,r_{v,\delta})\) are pairwise disjoint, patching these
local modifications together we obtain a map
\[
g^{(V)}\in C^1(\Omega;\mathbb R^2)
\]
such that
\[
g^{(V)}=g
\qquad\text{on } \Omega\setminus \bigcup_{v\in V} B(v,r_{v,\delta}).
\]
Moreover, since \(V\) is finite, choosing all the local parameters sufficiently small, we may ensure that
\[
\|g^{(V)}-g\|_{W^{2,1}(\Omega)}<\frac{\delta}{2},
\qquad
\|Dg^{(V)}-Dg\|_{L^\infty(\Omega)}<\frac{\delta}{2},
\]
and
\[
\det Dg^{(V)}\ge \frac{3\lambda}{4}
\qquad\text{on }\Omega.
\]
We now choose the final exceptional neighborhoods. Since $S\subset\overline\Omega$
is finite, we may choose pairwise disjoint open sets
\[
W_{s,\delta}\subset \mathbb R^2,\qquad s\in S,
\]
such that
\[
s\in W_{s,\delta},\qquad
\operatorname{diam}(W_{s,\delta})<\delta
\quad\text{for every }s\in S.
\]
Notice that the sets \(W_{s,\delta}\) are chosen as open subsets of
\(\mathbb R^2\), not necessarily as subsets of \(\Omega\), because some points
of \(S\) may lie on \(\partial\Omega\).

Moreover, we choose them so that:
\begin{enumerate}
\item if \(s\in\mathcal V\), then
\[
B(s,r_{s,\delta})\subset W_{s,\delta};
\]

\item for every interior edge \(e\in\mathcal E\), with endpoints
\[
\partial_{\mathrm{rel}}e=\{s_e^-,s_e^+\},
\]
the set
\[
e_\delta
:=
e\setminus
\left(
W_{s_e^-,\delta}\cup W_{s_e^+,\delta}
\right)
\]
is either empty or a compact line segment contained in the relative interior of \(e\).
\end{enumerate}

For each interior edge \(e\in\mathcal E\) such that
\(e_\delta\neq\varnothing\), choose an open parallelogram neighbourhood
\[
T_e=A_e(Q)\subset\Omega,\qquad Q=(-2,2)^2,
\]
where
\[
A_e(x)=M_ex+a_e
\]
is an orientation-preserving affine map, so small that:
\begin{enumerate}
\item the sets \(T_e\) are pairwise disjoint;

\item \(T_e\cap B(v,r_{v,\delta})=\varnothing\) for every \(v\in\mathcal V\);

\item \(A_e(\{x_1=0\})\subset e\), and \(A_e(Q^-)\), \(A_e(Q^+)\)
are contained in the two cells adjacent to \(e\);

\item
\[
A_e(\{0\}\times[-1,1])\supset e_\delta;
\]

\item
\[
T_e\setminus
\left(
W_{s_e^-,\delta}\cup W_{s_e^+,\delta}
\right)
\subset
A_e\bigl(Q\cap\{|x_2|<1\}\bigr).
\]
\end{enumerate}

Because the vertex smoothing is supported inside the disks \(B(v,r_{v,\delta})\), we have
\[
g^{(V)}=g
\qquad\text{on }T_e
\]
for every such edge \(e\).

Fix one of these edges \(e\), and define
\[
\widehat g_e:=g\circ A_e:Q\to\mathbb R^2.
\]
Then \(\widehat g_e\) is a two-cell piecewise quadratic map, \(C^1\)-compatible across the interface \(\{x_1=0\}\), and
\[
\det D\widehat g_e(x)
=
\det Dg(A_e x)\,\det M_e
\ge
\lambda\,\det M_e
=: \lambda_e>0
\qquad\text{on }Q.
\]
Hence Proposition \ref{prop:localized_edge_smoothing} applies to \(\widehat g_e\). Therefore, for every sufficiently small parameter \(\varepsilon_e>0\), there exists a map
\[
\widehat G_e\in C^1(Q;\mathbb R^2)
\]
such that:
\begin{enumerate}
\item \(\widehat G_e=\widehat g_e\) in a neighborhood of \(\partial Q\);
\item \(\widehat G_e\in C^\infty\bigl(Q\cap\{|x_2|<1\}\bigr)\);
\item
\[
\|\widehat G_e-\widehat g_e\|_{W^{2,1}(Q)}
+
\|D\widehat G_e-D\widehat g_e\|_{L^\infty(Q)}
\]
can be made arbitrarily small.
\end{enumerate}

Now define
\[
G_e:=\widehat G_e\circ A_e^{-1}:T_e\to\mathbb R^2.
\]
By Lemma \ref{lem}, \(G_e\in C^1(T_e;\mathbb R^2)\), \(G_e=g\) in a neighborhood of \(\partial T_e\), and
\[
G_e\in C^\infty\Bigl(A_e\bigl(Q\cap\{|x_2|<1\}\bigr)\Bigr).
\]
Moreover,
\[
\|G_e-g\|_{W^{2,1}(T_e)}
+
\|DG_e-Dg\|_{L^\infty(T_e)}
\]
can be made arbitrarily small by choosing \(\varepsilon_e\) sufficiently small.
Since
\[
\det Dg\ge \lambda>0
\qquad\text{on }T_e,
\]
the continuity of the determinant and the uniform convergence \(DG_e\to Dg\) on \(T_e\) imply that, after possibly decreasing \(\varepsilon_e\),
\[
\det DG_e\ge \frac{3\lambda}{4}
\qquad\text{on }T_e.
\]
Since the family of relevant edges is finite, choosing all parameters \(\varepsilon_e\) sufficiently small we may also require that
\[
\sum_{e\in E}\|G_e-g\|_{W^{2,1}(T_e)}<\frac{\delta}{2},
\qquad
\max_{e\in E}\|DG_e-Dg\|_{L^\infty(T_e)}<\frac{\delta}{2}.
\]

We now define the global map
\[
\widetilde g_\delta(x):=
\begin{cases}
G_e(x), & x\in T_e \text{ for some } e\in E,\\[4pt]
g^{(V)}(x), & x\in \Omega\setminus \displaystyle\bigcup_{e\in E}T_e.
\end{cases}
\]
This is well defined because the tubes \(T_e\) are pairwise disjoint. Moreover, by construction each \(G_e\) coincides with \(g\), hence with \(g^{(V)}\), in a neighborhood of \(\partial T_e\), since \(T_e\cap \bigcup_{v\in V}B(v,r_{v,\delta})=\varnothing\). Therefore
\[
\widetilde g_\delta\in C^1(\Omega;\mathbb R^2).
\]

We next prove (1) and (2). Since \(\widetilde g_\delta=g^{(V)}\) outside the union of the tubes and \(\widetilde g_\delta=G_e\) on each \(T_e\), we have
\[
\|\widetilde g_\delta-g\|_{W^{2,1}(\Omega)}
\le
\|g^{(V)}-g\|_{W^{2,1}(\Omega)}
+
\sum_{e\in E}\|G_e-g\|_{W^{2,1}(T_e)}
<\delta.
\]
Similarly,
\[
\|D\widetilde g_\delta-Dg\|_{L^\infty(\Omega)}
\le
\max\!\left\{
\|Dg^{(V)}-Dg\|_{L^\infty(\Omega)},
\max_{e\in E}\|DG_e-Dg\|_{L^\infty(T_e)}
\right\}
<\delta.
\]
Thus (1) and (2) hold.

We now prove (3). On \(\Omega\setminus \bigcup_{e\in E}T_e\) we have
\[
\widetilde g_\delta=g^{(V)},
\]
hence
\[
\det D\widetilde g_\delta=\det Dg^{(V)}\ge \frac{3\lambda}{4}.
\]
On each tube \(T_e\), we have
\[
\widetilde g_\delta=G_e,
\]
hence
\[
\det D\widetilde g_\delta=\det DG_e\ge \frac{3\lambda}{4}.
\]
Therefore
\[
\det D\widetilde g_\delta\ge \frac{3\lambda}{4}\ge \frac{\lambda}{2}
\qquad\text{on }\Omega,
\]
which proves (3).

We next establish (4). Set
\[
\phi_\delta:=\widetilde g_\delta-g.
\]
Then by (2),
\[
\|D\phi_\delta\|_{L^\infty(\Omega)}
=
\|D\widetilde g_\delta-Dg\|_{L^\infty(\Omega)}
<\delta.
\]
By Lemma 5.4, for every \(x,y\in\Omega\),
\[
|\phi_\delta(x)-\phi_\delta(y)|
\le Q_\Omega \|D\phi_\delta\|_{L^\infty(\Omega)}\,|x-y|
< Q_\Omega\delta\,|x-y|.
\]
Using Assumption 5.1(e),
\[
|g(x)-g(y)|\ge m|x-y|
\qquad\text{for all }x,y\in\Omega,
\]
we obtain
\[
\begin{aligned}
|\widetilde g_\delta(x)-\widetilde g_\delta(y)|
&\ge |g(x)-g(y)|-|\phi_\delta(x)-\phi_\delta(y)|\\
&\ge m|x-y|-Q_\Omega\delta\,|x-y|\\
&=(m-Q_\Omega\delta)|x-y|.
\end{aligned}
\]
This proves (4). In particular, if \(\delta<m/(2Q_\Omega)\), then
\[
|\widetilde g_\delta(x)-\widetilde g_\delta(y)|
\ge \frac{m}{2}|x-y|
\qquad\text{for all }x,y\in\Omega,
\]
so \(\widetilde g_\delta\) is injective.

It remains to prove (5). Let
$$
x\in \Omega\setminus\bigcup_{s\in S} \overline{W}_{s,\delta}
$$
There are two possibilities.

If \(x\in T_e\) for some interior edge \(e\), then by property (5) in the choice of \(T_e\),
\[
x\in T_e\setminus \bigl(W_{s_-,\delta}\cup W_{s_+,\delta}\bigr)
\subset A_e\bigl(Q\cap\{|x_2|<1\}\bigr).
\]
On this set one has
\[
\widetilde g_\delta = G_e \in C^\infty,
\]
by the construction of \(G_e\). Hence \(\widetilde g_\delta\) is smooth in a neighborhood of \(x\).

If
\[
x \notin \bigcup_{e\in E} T_e,
\]
then \(x\) lies outside all edge neighborhoods and outside all exceptional neighborhoods.
In particular, \(x\) does not belong to any interior edge: indeed, every point of an interior edge outside the exceptional neighborhoods is contained in the corresponding set \(T_e\) by the choice of the maps \(A_e\) and the inclusion
\[
A_e\bigl(\{0\}\times[-1,1]\bigr)\supset e_\delta.
\]
Hence \(x\) belongs to the interior of some cell of the partition, and in a neighborhood of \(x\) one has
\[
\widetilde g_\delta = g,
\]
which is quadratic there. Therefore \(\widetilde g_\delta\) is smooth in a neighborhood of \(x\).

Since \(x\) was arbitrary, we have proved that for every
\[
x\in \Omega\setminus \bigcup_{s\in S}\overline{W}_{s,\delta}
\]
there exists an open neighbourhood \(U_x\Subset \Omega\) such that
\[
\tilde g_\delta|_{U_x}\in C^\infty(U_x;\mathbb R^2).
\]
Equivalently,
\[
\tilde g_\delta \in C^\infty_{\mathrm{loc}}
\left(
\Omega \setminus \bigcup_{s\in S}\overline{W}_{s,\delta};
\mathbb R^2
\right).
\]
This proves (5).
\end{proof}
\begin{corollary}\label{cor:global_smoothing_sequence}
Under Assumption 5.1, there exist a sequence of maps
\[
\widetilde g_k \in C^1(\Omega;\mathbb R^2),
\qquad k\in\mathbb N,
\]
and, for every \(k\in\mathbb N\) and every \(s\in\mathcal S\), pairwise disjoint open sets
\[
W_{k,s}\subset \mathbb R^2,
\qquad s\in\mathcal S,
\]
such that:
\begin{enumerate}
\item
\[
s\in W_{k,s},\qquad
\operatorname{diam}(\overline{W}_{k,s})<\frac1k ;
\]
\item

\[
\tilde g_k\in C^\infty_{\mathrm{loc}}
\left(
\Omega\setminus\bigcup_{s\in S}\overline{W}_{k,s};
\mathbb R^2
\right).
\]
\item
\[
\widetilde g_k \to g
\qquad\text{in }W^{2,1}(\Omega;\mathbb R^2);
\]
\item
\[
D\widetilde g_k \to Dg
\qquad\text{uniformly on }\Omega;
\]
\item
\[
\det D\widetilde g_k \ge \frac{\lambda}{2}
\qquad\text{on }\Omega
\quad\text{for every }k\in\mathbb N;
\]
\item each \(\widetilde g_k\) is injective on \(\Omega\).
\end{enumerate}
\end{corollary}

\begin{proof}
Let \(Q_\Omega>0\) be the constant appearing in Theorem \ref{thm:global_smoothing}, and set
\[
\delta_k:=\min\left\{\frac1k,\frac{m}{4Q_\Omega}\right\},
\qquad k\in\mathbb N.
\]
Then \(\delta_k>0\), \(\delta_k\to0\), and
\[
\delta_k<\frac{m}{2Q_\Omega}
\qquad\text{for every }k\in\mathbb N.
\]

Applying Theorem \ref{thm:global_smoothing} with \(\delta=\delta_k\), we obtain, for every
\(k\in\mathbb N\), a map
\[
\widetilde g_k\in C^1(\Omega;\mathbb R^2)
\]
and pairwise disjoint open sets
\[
W_{k,s}\subset\mathbb R^2,
\qquad s\in\mathcal S,
\]
such that
\[
s\in W_{k,s},
\qquad
\operatorname{diam}(W_{k,s})<\delta_k\le \frac1k,
\]
and
\[
\widetilde g_k\in C^\infty\!\left(\Omega\setminus \bigcup_{s\in\mathcal S}W_{k,s}\right).
\]
This proves (1) and (2).

Moreover, by Theorem \ref{thm:global_smoothing},
\[
\|\widetilde g_k-g\|_{W^{2,1}(\Omega)}<\delta_k,
\qquad
\|D\widetilde g_k-Dg\|_{L^\infty(\Omega)}<\delta_k,
\]
hence, since \(\delta_k\to0\),
\[
\widetilde g_k\to g
\qquad\text{in }W^{2,1}(\Omega;\mathbb R^2),
\]
and
\[
D\widetilde g_k\to Dg
\qquad\text{uniformly on }\Omega.
\]
This proves (3) and (4).

Again by Theorem \ref{thm:global_smoothing},
\[
\det D\widetilde g_k\ge \frac{\lambda}{2}
\qquad\text{on }\Omega
\]
for every \(k\in\mathbb N\), proving (5).

Finally, since \(\delta_k<m/(2Q_\Omega)\) for every \(k\), the injectivity conclusion in
Theorem \ref{thm:global_smoothing} applies to each \(\widetilde g_k\). Therefore each
\(\widetilde g_k\) is injective on \(\Omega\), and (6) follows.
\end{proof}

\begin{remark}
We emphasize that the uniform convergence of the gradients
\[
\|D\widetilde g_k-Dg\|_{L^\infty(\Omega)}\to0
\]
stated in (iv) does not follow from abstract Sobolev embeddings, since
\(W^{2,1}(\Omega)\) is not continuously embedded into \(W^{1,\infty}(\Omega)\) in dimension two.
Rather, it is a strong feature of our specific approximation technique: the explicit use of
quadratic polynomials and smooth cut-off functions in Propositions 3.2 and 4.2 directly
provides pointwise uniform control on the first derivatives of the perturbation.
\end{remark}
\section{Approximation theorem in $W^{2,1}$}
\label{sec:approx}

The global smoothing theorem does not by itself produce a 
piecewise quadratic approximation of a general $W^{2,1}$ 
homeomorphism. We isolate the remaining step as a separate 
assumption.

\begin{assumption}
\label{ass:6.1}
Let $\Omega \subset \mathbb R^2$ be a bounded open connected polygonal domain,
and let $h\in C^1(\Omega;\mathbb R^2)$ be injective. Assume that
\[
  \det Dh \geq \mu > 0 \quad \text{on } \Omega,
\]
and there exist finitely many pairwise disjoint open sets
\(U_1,\ldots,U_N\) in \(\Omega\) such that
\[
h\in C^\infty_{\mathrm{loc}}
\left(
\Omega\setminus\bigcup_{i=1}^N \overline{U_i}^{\,\Omega};
\mathbb R^2
\right),
\]
where \(\overline{U_i}^{\,\Omega}\) denotes the closure of \(U_i\) relative to \(\Omega\).

Then, for every $\eta > 0$, there exists a map
$h_\eta \in C^\infty(\Omega;\mathbb{R}^2)$ such that:
\begin{enumerate}
  \item $h_\eta$ is injective on $\Omega$;
  \item $\|h_\eta - h\|_{W^{2,1}(\Omega)} < \eta$;
  \item $\|Dh_\eta - Dh\|_{L^\infty(\Omega)} < \eta$;
  \item $\det Dh_\eta \geq \mu/2$ on $\Omega$.
\end{enumerate}
\end{assumption}

\begin{theorem}
\label{thm:6.2}
Assume Assumption~\ref{ass:5.1} and Assumption~\ref{ass:6.1}.
Then there exists a sequence of maps
\[
  g_k \in C^\infty(\Omega;\mathbb{R}^2), \quad k \in \mathbb{N},
\]
such that:
\begin{enumerate}
  \item each $g_k$ is injective on $\Omega$;
  \item $g_k \to g$ in $W^{2,1}(\Omega;\mathbb{R}^2)$;
  \item $Dg_k \to Dg$ uniformly on $\Omega$;
  \item $\det Dg_k \geq \lambda/4$ on $\Omega$ for every $k \in \mathbb{N}$.
\end{enumerate}
\end{theorem}

\begin{proof}
By Corollary~\ref{cor:global_smoothing_sequence}, there exist maps
$\widetilde{g}_k \in C^1(\Omega;\mathbb{R}^2)$, $k \in \mathbb{N}$,
and, for every $k \in \mathbb{N}$ and every $s \in \mathcal S$,
pairwise disjoint open sets $W_{k,s} \subset \mathbb{R}^2$ such that:
\begin{itemize}
\item $
\tilde g_k\in C^\infty_{\mathrm{loc}}
\left(
\Omega\setminus\bigcup_{s\in S}\overline{W}_{k,s};
\mathbb R^2
\right);
$

  \item $\widetilde{g}_k \to g$ in $W^{2,1}(\Omega;\mathbb{R}^2)$;
  \item $D\widetilde{g}_k \to Dg$ uniformly on $\Omega$;
  \item $\det D\widetilde{g}_k \geq \lambda/2$ on $\Omega$
        for every $k \in \mathbb{N}$;
  \item each $\widetilde{g}_k$ is injective on $\Omega$.
\end{itemize}

We wish to apply Assumption 6.1 to each map \(\tilde g_k\).
For this purpose, set
\[
U_{k,s}:=W_{k,s}\cap\Omega,\qquad s\in S.
\]
Then each \(U_{k,s}\) is open in \(\Omega\), and the family
\(\{U_{k,s}\}_{s\in S}\) is pairwise disjoint. Since \(W_{k,s}\) is open in
\(\mathbb R^2\), we have
\[
\overline{U_{k,s}}^{\,\Omega}
=
\overline{W_{k,s}}\cap\Omega .
\]
Therefore
\[
\Omega\setminus\bigcup_{s\in S}\overline{U_{k,s}}^{\,\Omega}
=
\Omega\setminus\bigcup_{s\in S}\overline{W_{k,s}}.
\]
By Corollary 5.6,
\[
\tilde g_k\in C^\infty_{\mathrm{loc}}
\left(
\Omega\setminus\bigcup_{s\in S}\overline{W_{k,s}};
\mathbb R^2
\right),
\]
and hence
\[
\tilde g_k\in C^\infty_{\mathrm{loc}}
\left(
\Omega\setminus\bigcup_{s\in S}\overline{U_{k,s}}^{\,\Omega};
\mathbb R^2
\right).
\]

Thus Assumption 6.1 applies to \(\tilde g_k\) with
\[
h=\tilde g_k,\qquad
\mu=\frac{\lambda}{2},\qquad
\{U_i\}=\{U_{k,s}\}_{s\in S},\qquad
\eta=\frac1k.
\]
We obtain a map \(g_k\in C^\infty(\Omega;\mathbb R^2)\) such that
\[
g_k \text{ is injective on }\Omega,
\]
\[
\|g_k-\tilde g_k\|_{W^{2,1}(\Omega)}<\frac1k,
\]
\[
\|Dg_k-D\tilde g_k\|_{L^\infty(\Omega)}<\frac1k,
\]
and
\[
\det Dg_k\ge \frac{\lambda}{4}\quad\text{on }\Omega.
\]
To prove (2), we use the triangle inequality:
\[
  \|g_k - g\|_{W^{2,1}(\Omega)}
  \leq \|g_k - \widetilde{g}_k\|_{W^{2,1}(\Omega)}
     + \|\widetilde{g}_k - g\|_{W^{2,1}(\Omega)}
  < \frac{1}{k} + \|\widetilde{g}_k - g\|_{W^{2,1}(\Omega)}.
\]
The first term tends to $0$ since it is bounded by $1/k$, and the
second tends to $0$ by Corollary~\ref{cor:global_smoothing_sequence}. Hence
$g_k \to g$ in $W^{2,1}(\Omega;\mathbb{R}^2)$, which proves~(2).

To prove (3), we use the triangle inequality similarly:
\[
  \|Dg_k - Dg\|_{L^\infty(\Omega)}
  \leq \|Dg_k - D\widetilde{g}_k\|_{L^\infty(\Omega)}
     + \|D\widetilde{g}_k - Dg\|_{L^\infty(\Omega)}
  < \frac{1}{k} + \|D\widetilde{g}_k - Dg\|_{L^\infty(\Omega)}.
\]
The first term tends to $0$, and the second tends to $0$ by
Corollary 5.6. Therefore $Dg_k \to Dg$ uniformly
on $\Omega$, which proves~(3) and completes the proof.
\end{proof}
\begin{remark}
\label{rem:6.2}
For maps satisfying Assumption 5.1, the previous results reduce the remaining
analytic obstruction to a localized smoothing problem near the finite set of
endpoints of the interior edges of the partition. In particular, one obtains injective \(C^1\) approximants \(\tilde g_k\) which
are smooth outside arbitrarily small neighborhoods of this set and retain
quantitative lower bounds on both the Jacobian and the metric distortion.
Assumption 6.1 isolates the additional completion step needed to pass from this
localized result to a global smooth approximation theorem.
\end{remark}

\section*{Declarations}
\noindent \textbf{Funding.} The author is a member of the Gruppo Nazionale per l'Analisi Matematica, la Probabilit\`a e le loro Applicazioni (GNAMPA) of the Istituto Nazionale di Alta Matematica (INdAM). CUP E53C25002010001. \par\vspace{0.2cm}
\noindent \textbf{Competing Interests.} The author has no relevant financial or non-financial interests to disclose. \par\vspace{0.2cm}
\noindent \textbf{Data Availability.} Data sharing is not applicable to this article as no datasets were generated or analysed during the current study.

\end{document}